\documentclass[a4paper,11pt]{amsart}
\usepackage{amsfonts}
\usepackage{amssymb}
\usepackage[utf8]{inputenc}
\usepackage{amsmath}
\usepackage{pdflscape}
\usepackage{graphicx}
\usepackage[colorlinks=true, linkcolor=red, citecolor=blue]{hyperref}
\usepackage[]{epsfig}
\usepackage[]{pstricks}
\usepackage{tikz}

\usepackage[]{epsfig}
\usepackage[]{pstricks}
\usepackage{tikz}

\newtheorem{theorem}{Theorem}[section]

\newtheorem{lemma}[theorem]{Lemma}

\theoremstyle{definition}

\theoremstyle{remark}
\newtheorem*{remarks}{Remarks}

\newtheorem{claim}{Claim}

\newcommand{\remove}[1]{ }

\def\be{\begin{equation}}
\def\ee{\end{equation}}
\def\ba{\begin{eqnarray}}
\def\ea{\end{eqnarray}}

\numberwithin{equation}{section}
\begin{document}
\title[Control with integral condition for a dispersive system]{Control results with overdetermination condition for higher order dispersive system}
\author[Capistrano--Filho]{Roberto de A. Capistrano--Filho}
\address{Departamento de Matemática, Universidade Federal de Pernambuco (UFPE), 50740-545, Recife-PE, Brazil.}
\email{roberto.capistranofilho@ufpe.br}
\author[de Sousa]{Luan Soares de Sousa}
\email{luan.soares@ufpe.br}

\begin{abstract}
In recent years, controllability problems for dispersive systems have been extensively studied. This work is dedicated to proving a new type of controllability for a dispersive fifth order equation that models water waves, what we will now call the \textit{overdetermination control problem}. Precisely, we are able to find a control acting at the boundary that guarantees that the solutions of the problem under consideration satisfy an overdetermination integral condition. In addition, when we make the control act internally in the system, instead of the boundary, we are also able to prove that this condition is satisfied. These problems give answers that were left open in \cite{CaGo} and present a new way to prove boundary and internal controllability results for a fifth order KdV type equation.
\end{abstract}
\maketitle

\section{Introduction\label{Sec0}} 
\subsection{Setting of the problem} The Kawahara equation proposed in 1972 by T. Kawahara \cite{Kawahara} is a fifth-order Korteweg-de Vries equation (KdV) that can be viewed as a generalization of the KdV equation, which occurs in the theory of shallow water waves and take the form  \be\label{kaw} u_t+u_x+ u_{xxx}+ \alpha u_{xxxxx}+uu_x=0, \ee when $\alpha=-1$ and $u=u(t,x)$ is a real-valued function of two real variables $(t,x)$. It is important to point out that there are others physical background of Kawahara equation or in view of perturbed equation of KdV\footnote{Considering $\alpha=0$ in \eqref{kaw} we have the so-called KdV equation, for a historic review of this equation we can cite \cite{Bona1} and the reference therein.} and the authors suggest to reader see \cite{Boyd,Hunter,Pomeau}, among others.

\vspace{0.1cm}

In this article we will be interested with a kind of a control property to the Kawahara equation when an \textit{integral overdetermination condition} is required, namely
\begin{equation}\label{IO}
\int_{0}^{L}u(t,x)\omega(x)dx= \varphi(t), \ t \in [0,T],
\end{equation}
with some known functions $\omega$ and $\varphi$. To present the problem, let us consider the Kawahara equation in the bounded rectangle $Q_T = (0, T) \times (0,L)$, where $T$ and $L$ are positive numbers with boundary function $h_i$, for $i=1,2,3,4$ and $h$ or the right-hand side $f$ of a special form to specify latter, namely, 
\begin{equation}\label{int1}
\left\lbrace
\begin{array}{llr}
u_{t} + u_{x} +u_{xxx} - u_{xxxxx}+uu_x= f(t,x) & \mbox{in} \ Q_{T}, \\ 
u(t,0)=h_{1}(t), \ u(t,L)= h_{2}(t), \ u_{x}(t,0)= h_{3}(t),& \mbox{in} \ [0,T], \\
u_{x}(t,L)=h_{4}(t), \ u_{xx}(t,L) = h(t) & \mbox{in} \ [0,T],\\
 u(0,x) = u_{0}(x) & \mbox{in} \ [0,L].
\end{array}\right. 
\end{equation} 

Thus, we are interested in studying two control problems, which we will call them from now on by \textit{overdetermination control problem}. The first one can be read as follows:

\vspace{0.2cm}
\noindent\textbf{Problem $\mathcal{A}$:} For given functions $u_0$, $h_i$, for $i=1,2,3,4$ and $f$ in some appropriated spaces, can we find a boundary control $h$ such that the solution associated to the equation \eqref{int1} satisfies the integral overdetermination \eqref{IO}?
\vspace{0.2cm}

The second problem of this work is concentrated to prove that for a special form of the function 
\begin{equation}\label{f}
f(t,x)=f_0(t)g(t,x), \quad (t,x)\in Q_T,
\end{equation}  
the integral overdetermination \eqref{IO} is verified, in other words.

\vspace{0.2cm}
\noindent\textbf{Problem $\mathcal{B}$:} For given functions $u_0$, $h_i$, for $i=1,2,3,4$, $h$ and $g$ in some appropriated spaces, can we find a internal control $f_0$ such that the solution associated to the equation \eqref{int1} satisfies the integral overdetermination \eqref{IO}?
\vspace{0.2cm}

Therefore, the main purpose of this paper is to prove that these problems are indeed true. There are basically two features to be emphasized in this way:
\begin{itemize}
\item One should be convinced that the integral overdetermination condition is effective and gives good (internal and boundary) control properties. In fact, this kind of condition is very important in the inverse problem (see e.g. \cite{PriOrVas}) and it is a new way of controlling dispersive systems. 
\vspace{0.1cm}
\item One should be capable of controlling the system when the control acts in $[0,T]$, which is also new for the Kawahara equation (see for instance \cite{CaGo} for details of internal control problems for Kawahara equation).
\end{itemize}

\subsection{Bibliographical comments} We comment briefly on the bibliography emphasizing the works related with the \textit{well-posedness} and \textit{controllability}. Before presenting it, we caution that this is only a small sample of the extant works existent for the Kawahara equation since there are other subjects of interest from a mathematical point of view.

\subsubsection{Well-posedness results} Regarding the Cauchy problem some authors showed the local and global well-posedness results. For example, Kenig \textit{et al.} \cite{kepove} proved the well-posedness result for a general nonlinear dispersive equation, which one with some restrictions, can be viewed as \eqref{kaw}. In this celebrated work, the authors are able to prove that the associated initial value problem (IVP) is locally well-posed in weighted Sobolev spaces. We would like to mention that in \cite{CaGo,khanal} the authors also treated the theory of well-posedness in weighted Sobolev spaces for the Kawahara equation. Recently, Cui \textit{et al.} \cite{Cui} studied the Cauchy problem of the Kawahara equation  in $L^2$-space, precisely, they proved the global well-posedness for \eqref{kaw}. Considering the initial boundary value problem (IBVP) we can see relevant advances in \cite{Doronin2}, for homogeneous boundary conditions, and in \cite{FaOp}, for the half-line. In addition to these works, some other works treat the well-posedness theory, we can cite, for example, \cite{Faminskii,ponce}.

\subsubsection{Controllability results} As is well known the control theory can be studied in two ways: Stabilization problem and internal or boundary control problems (see \cite{zhang1,CaKawahara} for details of these kinds of issues).

In this spirit, let us start to mention a pioneer work concerning the stabilization property for the Kawahara equation. In \cite{CaKawahara}, the first author with some collaborators were able to introduce an internal feedback law in \eqref{int1}, considering the nonlinearity $u^2u_x$ instead of $uu_x$ and $h(t)=h_i(t)=0$, for $i=1,2,3,4$. Being precise, they proved that under the effect of the damping mechanism the energy associated with the solutions of the system decays exponentially. Additionally, they conjecture the existence the existence of an important phenomena, is the so-called \textit{critical set phenomenon} as occurs with the single  KdV equation \cite{CaZh,Rosier} and the Boussinesq KdV-KdV system \cite{CaPaRo1}\footnote{Differently what happens with KdV and Boussinesq KdV-KdV the characterization of the critical set for the Kawahara equation is an open issue, we cite \cite{Vasconcellos} for details of this subject.}. We also would like to suggest to the reader the reference \cite{Doronin} to stabilization problems related to the Kawahara equation in the real line.

Now, some references of internal control problems are presented. This problem was first addressed in \cite{zhang} and after that in \cite{zhang1}. In both cases the authors considered the Kawahara equation in a periodic domain $\mathbb{T}$ with a distributed control of the form \[ f(t,x)=(Gh)(t,x):= g(x)(  h(t,x)-\int_{\mathbb{T}}g(y) h(t,y) dy), \] where $g\in C^\infty (\mathbb T)$ supported in $\omega\subset\mathbb{T}$ and $h$ is a control input. Here, it is important to observe that the control in consideration has a different form as presented in \eqref{f}, and the result is proven in a different direction from what we will present in this manuscript.

Still related with internal control issues, Chen \cite{MoChen} presented results considering the Kawahara equation \eqref{int1} posed on a bounded interval with a distributed control $f(t,x)$ and homogeneous boundary conditions. She showed the result taking advantage of a Carleman estimate associated to the linear operator of the Kawahara equation with an internal observation. With this in hand, she was able to get a null controllable result when $f$ is effective in a $\omega\subset(0,L)$.  As the results obtained by her do not answer all the issues of the internal controllability, in a recent article \cite{CaGo} the authors closed some gaps left in \cite{MoChen}. Precisely, considering the system \eqref{int1} with an internal control $f(t,x)$ and homogeneous boundary conditions, the authors are able to show that the equation in consideration is exact controllable in $L^2$-weighted Sobolev spaces and, additionally, the Kawahara equation is controllable by regions on $L^2$-Sobolev space, for details see \cite{CaGo}.

Finally, related to the boundary control problem, there is a unique result which one was proved in \cite{GG}. The authors consider the boundary conditions as in \eqref{int1} and show that exact controllability holds when two or until five controls are inputting in these boundary conditions.  

\subsection{Notations and Main results} With these previous results in hand, we are able to present our main results that tries to answer questions left open in the manuscript \cite{CaGo} and presents an alternative way for the boundary control problems of the Kawahara equation. First of all, let us introduce the following notation that we will use in the article from now on.
\begin{itemize}
\item[i.] Denote by $$X(Q_{T})= C([0,T]; L^{2}(0,L)) \cap L^{2}(0,T; H^{2}(0,L)),$$ the space equipped with the following norm
$$\|v\|_{X(Q_{T})} = \displaystyle\max_{t \in [0,T]}\|v(t,\cdot)\|_{L^{2}(0,L)} + \|v_{xx}\|_{L^{2}(Q_{T})}= \displaystyle\|v\|_{C([0,T];L^{2}(0,L))} + \|v_{xx}\|_{L^{2}(Q_{T})}.$$
\item[ii.] Consider $$\mathcal{H}= H^{\frac{2}{5}}(0,T)\times H^{\frac{2}{5}}(0,T)\times H^{\frac{1}{5}}(0,T)\times H^{\frac{1}{5}}(0,T),$$ with the norm
$$ \|\widetilde{h}\|_{\mathcal{H}}= \|h_{1}\|_{H^{\frac{2}{5}}(0,T)} + \|h_{2}\|_{H^{\frac{2}{5}}(0,T)} +\|h_{3}\|_{H^{\frac{1}{5}}(0,T)} +\|h_{4}\|_{H^{\frac{1}{5}}(0,T)},$$
where $\widetilde{h}=(h_1,h_2,h_3,h_4)$.
\item[iii.] The intersection $(L^{p} \cap L^{q})(0,T)$ will be considered with the following norm  $$\|\cdot ,\cdot \|_{(L^{p} \cap L^{q})(0,T)}= \|\cdot ,\cdot \|_{L^{p}(0,T)} + \|\cdot , \cdot \|_{L^{q}(0,T)}.$$ 
\item[iv.] Finally, for any $p \in [1, \infty]$, we denote by  $$\widetilde{W}^{1,p}(0,T)= \{\varphi \in W^{1,p}(0,T); \varphi(0)= 0\},$$ with the norm defined by 
$$ \|\varphi\|_{\widetilde{W}^{1,p}(0,T)}= \|\varphi'\|_{L^{p}(0,T)}.$$
\item[vi.] Consider $\omega$ be a fixed function which belongs to the following set 
\begin{equation}\label{jota}
\mathcal{J}= \{\omega \in H^{5}(0,L)\cap H^{2}_{0}(0,L); \ \omega''(0)=0 \}.
\end{equation}
\end{itemize}

The first result of the manuscript gives us an answer for the Problem $\mathcal{A}$, presented in the beginning of the introduction. The answer for the boundary overdetermination control problem for the system \eqref{int1} can be read as follows.

\begin{theorem}\label{main} Let $p \in [2,\infty]$. Suppose that $u_{0} \in L^{2}(0,L)$,  $f \in  L^{p}(0,T; L^{2}(0,L))$, $\widetilde{h} \in \mathcal{H}$ and $h_i\in L^p(0,T)$, for $i=1,2,3,4$. If $\varphi \in L^{p}(0,T)$ and $\omega \in \mathcal{J}$ are such that  $\omega''(L)\neq 0$
and
\begin{equation}\label{OIa}
\int_{0}^{L}u_{0}(x)\omega(x)dx= \varphi(0),
\end{equation}
considering $c_{0}= \|u_{0}\|_{L^{2}(0,L)} + \ \|f\|_{L^{2}(0,T; L^{2}(0,L))} + \ \|\widetilde{h}\|_{\mathcal{H}} + \ \|\varphi'\|_{L^{2}(0,T)}$, the following assertions hold true.
\begin{itemize}
\item[1.] For a fixed $c_{0}$, there exists $T_{0} > 0$ such that for $T \in (0,T_{0}]$, then we can find a unique function $h \in L^{p}(0,T)$ in such a way that the solution $u\in X(Q_{T})$ of \eqref{int1} satisfies \eqref{IO}.

\vspace{0.2cm}

\item[2.]For each $T > 0$ fixed, exists a constant $\gamma > 0$ such that for $c_{0} \leq \gamma,$ then we can find a unique boundary control $h \in L^{p}(0,T)$ with the solution $u \in X(Q_{T})$ of \eqref{int1} satisfying \eqref{IO}.
\end{itemize}
\end{theorem}

The next result ensures for the first time that we are able to control the Kawahara equation with a function $f_0$ supported in $[0,T]$. Precisely, we will respond to the Problem $\mathcal{B}$ mentioned in this introduction.

\begin{theorem}\label{main1}
Let $p \in [1,\infty]$, $u_{0} \in L^{2}(0,L)$,  $h \in  L^{max\{2,p\}}(0,T; L^{2}(0,L))$, $ \widetilde{h} \in \mathcal{H}$ and $h_i\in L^p(0,T)$, for $i=1,2,3,4$. If $\varphi \in L^{p}(0,T)$, $ g \in C([0,T]; L^{2}(0,L))$ and $\omega \in \mathcal{J}$ are such that $\omega''(L)\neq 0,$ and there exists a positive constant $g_0$ such that \eqref{OIa} is satisfied and $$\left|\int_{0}^{L}g(t,x)\omega(x)dx\right| \geq g_{0} > 0,$$
considering $c_{0}= \|u_{0}\|_{L^{2}(0,L)} + \ \|h\|_{L^{2}(0,L)} + \ \|\widetilde{h}\|_{\mathcal{H}} + \ \|\varphi'\|_{L^{1}(0,T)}$, we have that:
\begin{itemize}
\item[1.] For a fixed $c_{0}$, so there exists $T_{0} > 0$ such that for $T \in (0,T_{0}]$, exists a unique $f_{0} \in L^{p}(0,T)$ and a solution $u \in X(Q_{T})$ of  \eqref{int1}, with $f$ defined by \eqref{f}, satisfying \eqref{IO}.

\vspace{0.2cm}

\item[2.] For a fixed $T > 0$, there exists a constant $\gamma > 0$ such that for $c_{0} \leq \gamma,$ we have the existence of a control input  $f_{0} \in L^{p}(0,T)$ which the solution $u \in X(Q_{T})$ of \eqref{int1}, with $f$ as in \eqref{f},  verifies \eqref{IO}.
\end{itemize}
\end{theorem}

\subsection{Heuristic of the article and further comments}In this article, we investigate and discuss overdetermination control problems with respect to boundary and internal variations. As can be seen in this introduction, the agenda of the research of control theory for the fifth order KdV equation is quite new and does not acknowledge many results in the literature. With this proposal to fill this gap, we intend to present a new way to prove internal and boundary control results for this system. Thus, for this type of \textit{integral overdetermination condition} the first results on the solvability of control problems for the Kawahara equation are obtained in the present paper.

\subsubsection{Heuristic of the article} The first result is concerning of the \textit{boundary overdetermination control problem}, roughly speaking, we are able to find an appropriate control $h$, acting on the boundary term $u_{xx}(t,L)$, such that integral condition \eqref{IO} it turns out. Theorem \ref{main} is first proved for the linear system associated to \eqref{int1} and after that, using a fixed point argument, extended to the nonlinear system. The main ingredients are the Lemmas \ref{1.1} and \ref{lema3}. In the Lemma \ref{lema3} we are able to find two appropriate applications that links the boundary control term $h(t)$ with the overdetermination condition \eqref{IO}, namely
\begin{equation*}
\begin{array}{lll}
\Lambda:  L^{p}(0,T)&\longrightarrow \widetilde{W}^{1,p}(0,T)\\
\quad \quad  \quad h &\longmapsto (\Lambda h)(\cdot)= \displaystyle \int_{0}^{L}u(\cdot,x)\omega(x)dx
\end{array}
\end{equation*}
and
\begin{equation*}
\begin{array}{lll}
A: L^{p}(0,T) &\longrightarrow L^{p}(0,T)\\
\quad \quad \quad h &\longmapsto \displaystyle (Ah)(t)= \varphi'(t) - \int_{0}^{L}u(t,x)(\omega'(x) + \omega'''(x) - \omega''''''(x))dx, \quad \forall t \in [0,T].
\end{array}
\end{equation*}
So, we  prove that such application $\Gamma$ has an inverse which one is continuous, by Banach's theorem, showing the lemma in question, and so, reaching our goal, to prove Theorem \ref{main}.

\vspace{0.1cm}

Theorem \ref{main1} follows the same idea, the strictly different point is related with the appropriated applications which in this case links the internal control $f_0$ with the overdetermination condition \eqref{IO} (see Lemma \ref{lss}), in this case,  defined as follows
$$ (\Lambda f_{0})(\cdot)= \int_{0}^{L}u(\cdot,x)\omega(x)dx
$$
and
$$(Af_{0})(t)= \frac{\varphi'(t)}{g_{1}(t)} - \frac{1}{g_{1}(t)}\int_{0}^{L}u(t,x)(\omega' + \omega''' - \omega''''')dx,
$$
where, $$ g_{1}(t)= \int_{0}^{L}g(t,x)\omega(x)dx.$$

\subsubsection{Further comments}To conclude this introduction, we outline additional comments. It is important to point out that the method used here is commonly applied to inverse problems in optimal control. For the readers we cite this excellent book \cite{PriOrVas} for details of the integral conditions applied in inverse problems. 

\vspace{0.1cm}

With respect the generality of the work, we have the following points: 

\begin{itemize}
\item Theorems \ref{main} and \ref{main1} can be obtained for more general nonlinearities. Indeed, if we consider $v \in X(Q_{T})$ and $p \in (2,4]$, we have that 
$$
\int_{0}^{T} \int_{0}^{L} |v^{p+2} |d x d t\leqslant C\left\|v\right\|_{C([0,T];L^{2}(0,L))}^{p} \int_{0}^{T}\left\|v_{x}\right\|^{2} d t \leqslant C\left\|v\right\|_{X(Q_{T})}^{p + 2},
$$
by the Gagliardo–Nirenberg inequality. Moreover, recently, Zhou \cite{Zhou} showed the well-posedness of the following initial boundary value problem
\begin{equation}\label{int2} 
\left\{\begin{array}{lll}
u_t-u_{xxxxx}=c_{1} uu_x+c_{2} u^{2}u_x+b_{1}  u_xu_{xx}+b_{2} uu_{xxx},& x \in(0, L), \ t \in \mathbb{R}^{+}, \\
u(t, 0)=h_{1}(t), \quad u(t, L)=h_{2}(t),  \quad u_x(t, 0)=h_{3}(t), &  t \in \mathbb{R}^{+},\\
 u_x(t, L)=h_{4}(t), \quad u_{xx}(t, L)=h(t), &  t \in \mathbb{R}^{+},\\
u(0, x)=u_{0}(x),& x \in(0, L),
\end{array}\right.
\end{equation}
Thus, due to the previous inequality and the results proved in \cite{Zhou}, when we consider $b_1=b_2=0$ and the combination $c_1uu_{x}+c_2u^2u_x$ instead of $uu_x$ on \eqref{int1}, Theorems \ref{main} and \ref{main1} remains valid, however, with sake of simplicity, we consider only the nonlinearity as $uu_x$. 
\vspace{0.1cm}
\item Note that, in this manuscript, the regularity of the boundary terms are sharp in $H^s(0,L)$, for $s\geq0$. In fact, due the method introduced by Bona \textit{et al.} \cite{Bona1} for KdV equation the authors in \cite{zhang2} and \cite{Zhou} are able to provide sharp regularity for the traces function in both IBVP \eqref{int1} and \eqref{int2}. So, in this sense, Theorems \ref{main} and \ref{main1} gives a sharp regularity of the functions involved.
\vspace{0.1cm}
\item Unlike what happens in the case of the control problem considered in \cite{Rosier} for KdV equation, in \cite{CaPaRo1} for Boussinesq KdV--KdV equation and which was conjectured by the first author in \cite{CaKawahara}, here, due the method used, we can take $h_i = 0$, for $i=1,2,3,4$, and only consider a control acting in the trace $u_{xx}(t,L)$, without concern with the critical set phenomenon.
\vspace{0.1cm}
\item The arguments presented in this work have prospects to be applied for other nonlinear dispersive equations in the context of the bounded domains. In fact, our motivation was due to the fact that Faminskii \cite{Fa1} proved a result for the KdV equation, that is, when considering the system \eqref{kaw}  with $\alpha=0$. However, note that in \cite{Faminskii} the author decides to use the solution in a weak sense, ensuring that the results are verified for the function $\frac{u^2}{2}$, but, in our case, we can deal with more general the terms like $\frac{u^2}{2}$, $uu_x$ and $u^2u_x$.
\vspace{0.1cm}
\item Finally, this work presents another way to prove control results for the higher order dispersive system which are completely different from what was presented in \cite{CaGo,GG,zhang1}.
\end{itemize}

\subsection{Outline of the work} Section \ref{Sec1} is devoted to review the main results of the well-posedness for the fifth order KdV equation in Sobolev spaces. In the Section \ref{Sec2} we present two auxiliary lemmas which help us to prove the controllability results. The overdetermination control results, when the control is acting in the boundary and internally, are presents in the Sections \ref{Sec3} and \ref{Sec4}, respectively, that is, we will present the proof of the main results of the manuscript, Theorems \ref{main} and \ref{main1}.

\section{A fifth order KdV equation: A review of well-posedness results\label{Sec1}}
In this section let us treat the well-posedness of the fifth order KdV equation, that is, we are interested in the well-posedness of following system  
\begin{equation}\label{IBVP}
\left\lbrace
\begin{array}{llr}
u_{t} + u_{x} +u_{xxx} - u_{xxxxx}+uu_x= f(t,x) & \mbox{in} \ Q_{T}, \\ 
u(t,0)=h_{1}(t), \ u(t,L)= h_{2}(t), \ u_{x}(t,0)= h_{3}(t),& \mbox{in} \ [0,T], \\
u_{x}(t,L)=h_{4}(t), \ u_{xx}(t,L) = h(t) & \mbox{in} \ [0,T],\\
 u(0,x) = u_{0}(x) & \mbox{in} \ [0,L],
\end{array}\right. 
\end{equation} 
where $L, T >0$  are fixed real numbers, $Q_{T}= [0,T]\times[0,L]$ and $u_{0}, h_{1}, h_{2},h_{3}, h_{4}, h$ and $f$ are well-known functions. Precisely, we will put together the mains results of well-posedness to \eqref{IBVP}. 

\subsection{Homogeneous case} The first result is due to the first author \cite[Lemma 2.1]{CaKawahara} and provided the well-posedness results for the linear problem
\begin{equation}\label{h2}
\left\lbrace
\begin{array}{llr}
u_{t} + u_{x} +u_{xxx} - u_{xxxxx}+uu_x= 0& \mbox{in} \ Q_{T}, \\ 
u(t,0)=u(t,L)=u_{x}(t,0)= u_{x}(t,L)= u_{xx}(t,L) = 0 & \mbox{in} \ [0,T],\\
 u(0,x) = u_{0}(x) & \mbox{in} \ [0,L].
\end{array}\right. 
\end{equation} 

\begin{lemma}
\label{H1}Let $u_{0}\in L^{2}\left(  0,L\right)  $. Then (\ref{h2}) possesses 
a unique (mild) solution $u\in X(Q_T)$ with
\[
u_{xx}\left(  0,t\right)  \in L^{2}\left(  0,T\right)  .
\]
Moreover, there exists a constant $C=C\left(  T,L\right)  >0$ such that%
\begin{equation*}
\left\Vert u\right\Vert _{C^{0}\left(  \left[  0,T\right]  ;L^{2}\left(
0,L\right)  \right)  }+\left\Vert u\right\Vert _{L^{2}\left(  0,T;H^{2}\left(
0,L\right)  \right)  }\leq C\left\Vert u_{0}\right\Vert 
\end{equation*}
and%
\begin{equation*}
\left\Vert u_{xx}\left(  0,t\right)  \right\Vert _{L^{2}\left(  0,T\right)
}\leq\left\Vert u_{0}\right\Vert \text{.} 
\end{equation*}
\end{lemma}

The proof of this lemma is a direct consequence of semigroup theory and multipliers method. In the way to prove global well-posedness results for the nonlinear system, in \cite[Lemma 2.2 and 2.3]{CaKawahara}, the authors are able to prove some results for the following system 
\begin{equation}\label{h1}
\left\lbrace
\begin{array}{llr}
u_{t} + u_{x} +u_{xxx} - u_{xxxxx}+u^pu_x= 0& \mbox{in} \ Q_{T}, \\ 
u(t,0)=u(t,L)=u_{x}(t,0)= u_{x}(t,L)= u_{xx}(t,L) = 0 & \mbox{in} \ [0,T],\\
 u(0,x) = u_{0}(x) & \mbox{in} \ [0,L],
\end{array}\right. 
\end{equation} 
with $p\in(2,4]$. The global well-posedness for this system can be read as follows (we infer the read see \cite[Lemmas 2.2 and 2.3 and Remark 2.1]{CaKawahara} for details).

\begin{lemma}Let $T_{0}>0$ and $u_{0}\in L^{2}\left(  0,L\right)  $ be given.
Then there exists $T\in\left(  0,T_{0}\right]  $ such that (\ref{h1})
possesses a unique solution $u(t,x)\in Q_T$. Moreover, if
$\left\Vert u_{0}\right\Vert \ll1$, then
\begin{equation*}
\left\Vert u\right\Vert _{L^{2}\left(  0,T;H^{2}\left(  0,L\right)  \right)
}^{2}\leq c_1\left\Vert u_{0}\right\Vert ^{2} \left(1+\left\Vert u_{0}\right\Vert ^{4}\right)\text{,} 
\end{equation*}
where $c_{1}=c_{1}\left(  T,L\right)  $ is a positive constant. Moreover,
\begin{equation*}
u_{t}\in L^{4/3}\left(  0,T;H^{-3}\left(  0,L\right)  \right)  .
\end{equation*}
\end{lemma}

\subsection{Non-homogeneous case} For the nonhomogeneous initial-boundary value problem (IBVP)
\begin{equation}\label{Kwa}
\left\lbrace
\begin{array}{llr}
u_{t} + u_{x} +u_{xxx} - u_{xxxxx}= f(t,x) & \mbox{in} \ Q_{T}, \\ 
u(t,0)=h_{1}(t), \ u(t,L)= h_{2}(t), \ u_{x}(t,0)= h_{3}(t),& \mbox{in} \ [0,T], \\
u_{x}(t,L)=h_{4}(t), \ u_{xx}(t,L) = h(t) & \mbox{in} \ [0,T],\\
 u(0,x) = u_{0}(x) & \mbox{in} \ [0,L],
\end{array}\right. 
\end{equation} 
Zhao and Zhang \cite[Lemma 3.1]{zhang2} showed the following result:

\begin{lemma}\label{ZH} Let $T>0$ be given, there is a $C > 0$  such that for any  $f \in L^{1}(0,T; L^{2}(0,L))$, $u_{0}, h \in L^{2}(0,L)$ and $\widetilde{h} \in \mathcal{H}$, IBVP \eqref{Kwa} admits a unique solution (mild) $u:=S(u_0,h,f,\tilde{h}) \in X(Q_{T})$ satisfying
\begin{equation*}
\|u\|_{X(Q_{T})} \leq C \left( \|u_{0}\|_{L^{2}(0,L)} + \|h\|_{L^{2}(0,L)} + \|\widetilde{h}\|_{\mathcal{H}} + \|f\|_{L^{1}(0,T; L^{2}(0,L))}\right).
\end{equation*}
\end{lemma}

Considering the full system \eqref{Kw}, in this same work, Zhao and Zhang \cite[Lemma 3.2]{zhang2}, showed the following result. 

\begin{lemma}\label{Zhang 2}  There exists a constant $C>0$ such that for any $T>0$ and $u, v \in X(Q_{T})$ satisfying the following inequalities:
$$
\int_{0}^{T}\|uv_{x}\|_{L^2(0,L)}dt \leq C(T^{\frac{1}{2}} + T^{\frac{1}{4}})\|u\|_{X(Q_(T)}\|v\|{X(Q_(T)}
$$
and 
$$
\|uv_{x}\|_{W^{0,1}(0,T;L^{2}(0,L))} \leq C(T^{\frac{1}{2}} + T^{\frac{1}{4}})\|u\|_{X(Q_(T)}\|v\|_{X(Q_(T)}.
$$
\end{lemma}

\section{Auxiliary results} \label{Sec2}
In this section we are interested to prove some auxiliary lemmas for the solutions of the  system
\begin{equation}\label{Kw}
\left\lbrace
\begin{array}{llr}
u_{t} + u_{x} +u_{xxx} - u_{xxxxx}= f(t,x) & \mbox{in} \ Q_{T}, \\ 
u(t,0)=h_{1}(t), \ u(t,L)= h_{2}(t), \ u_{x}(t,0)= h_{3}(t),& \mbox{in} \ [0,T], \\
u_{x}(t,L)=h_{4}(t), \ u_{xx}(t,L) = h(t) & \mbox{in} \ [0,T],\\
 u(0,x) = u_{0}(x) & \mbox{in} \ [0,L].
\end{array}\right. 
\end{equation} 
To do this, consider $\omega\in\mathcal{J}$ defined by \eqref{jota} and define  $q:[0,T] \longrightarrow \mathbb{R}$ as follows
$$q(t)= \int_{0}^{L}u(t,x)\omega(x)dx,$$
where $u= S(u_{0}, h, f, \widetilde{h})$  is solution of \eqref{Kw} guaranteed by Lemma \ref{ZH}. The next two auxiliary lemmas are the key point to show the main results of this work. The first one, gives that $q\in W^{1,p}(0,L)$ and can be read as follows.

\begin{lemma}\label{1.1} Let $p \in [1, \infty]$ and the assumptions of Lemma \ref{ZH} be satisfied. Suppose that $h_{i}$, for $i=1,2,3,4$, and $h$ belonging in $L^{p}(0,T)$, $f= f_{1} + f_{2x}$, where $f_{1} \in L^{p}(0,T; L^{2}(0,L))$ and $f_{2} \in L^{p}(0,T; L^{1}(0,L))$. If $u= S(u_{0}, h, f_{1} + f_{2x}, \widetilde{h})$ is a mild solution of \eqref{Kw} and $\omega \in \mathcal{J}$, then the function $q\in W^{1,p}(0,T)$ and the relation
\begin{equation}\label{q'}
\begin{split}
 q'(t) =  &\ \omega''(L)h(t) - \omega'''(L)h_{4}(t) + \omega'''(0)h_{3}(t) + \omega''''(L)h_2(t) - \omega''''(L)h_{1}(t)  \\
 &+ \int_{0}^{L}f_{1}(t,x)\omega(x)dx - \int_{0}^{L}f_{2}(t,x)\omega'(x)dx+ \int_{0}^{L}u(t,x)[\omega'(x) +\omega'''(x) - \omega'''''(x)]dx 
\end{split}
\end{equation} 
holds for almost all $t\in[0,T]$. In addition, the function $q'\in L^{p}(0,T)$ can be estimate in the following way
\begin{equation}\label{norma de q'}
\begin{split}
\|q'\|_{L^{p}(0,T)} \leq &\ C\left( \|u_{0}\|_{L^{2}(0,L)} + \|h\|_{(L^{p}\cap L^{2})(0,T)} + \|h_{3}\|_{(L^{p}\cap H^{\frac{1}{5}})(0,T)}\right. \\
&+ \|h_{4}\|_{(L^{p}\cap H^{\frac{1}{5}})(0,T)} + \|h_{1}\|_{(L^{p}\cap H^{\frac{2}{5}})(0,T)}+\|h_{2}\|_{(L^{p}\cap H^{\frac{2}{5}})(0,T)}\\
&\left.+ \|f_{1}\|_{L^{p}(0,T;L^{2}(0,L))} +  \|f_{2}\|_{L^{p}(0,T;L^{1}(0,L))} + \|f_{2x}\|_{L^{1}(0,T;L^{2}(0,L))} \right),
\end{split}
\end{equation} 
with $C > 0$  a constant that is nondecreasing with increasing $T$.
\end{lemma}
\begin{proof}
Considering $\psi \in C_{0}^{\infty}(0,T)$, multiplying \eqref{Kw} by $\psi\omega$ and integrating by parts in $Q_{T}$ we have that
\begin{equation*}
\begin{split}
\int_{0}^{T}\psi'(t)q(t)dt=&\int_{0}^{T}\psi(t)\left( \omega''(L)h(t) - \omega'''(L)h_{4}(t) +\omega'''(0)h_{3}(t) \right.\\
&+\int_{0}^{T}\omega''''(L)h_2(t) - \omega''''(0)h_{1}(t)\\
&+\left.\int_{0}^{L}f_{1}(t,x)\omega(x)dx - \int_{0}^{L}f_{2}(t,x)\omega'(x)dx\right)dt\\
&+ \int_{0}^{L}u(t,x)(\omega'(x) + \omega'''(x) - \omega'''''(x))dx\\
=&-\int_0^T\psi(t)r(t)dt
\end{split}
\end{equation*} 
with $ r: [0,T] \longmapsto \mathbb{R}$ defined by
\begin{equation*}
\begin{split}
r(t) =  &\ \omega''(L)h(t) - \omega'''(L)h_{4}(t) + \omega'''(0)h_{3}(t) + \omega''''(L)h_2(t) - \omega''''(L)h_{1}(t)  \\
 &+ \int_{0}^{L}f_{1}(t,x)\omega(x)dx - \int_{0}^{L}f_{2}(t,x)\omega'(x)dx+ \int_{0}^{L}u(t,x)[\omega'(x) +\omega'''(x) - \omega'''''(x)]dx ,
\end{split}
\end{equation*}
which gives us $q'(t)=r(t)$.

It remains for us to prove that $q' \in L^{p}(0,T)$, for $p\in[1,\infty]$. To do it, we need to bound each term of \eqref{q'}. We will split the proof in two cases, namely, $p\in[1,\infty)$ and $p=+\infty$.

\vspace{0.2cm}
\noindent\textbf{Case 1.} $1 \leq p < \infty$. 
\vspace{0.2cm}

First, note that 
\begin{equation*}
\begin{split}
\left| \int_{0}^{L}u(t,x)(\omega'(x)+ \omega'''(x) - \omega'''''(x))dx\right|&\leq\|\omega\|_{H^{5}(0,L)}\|u(t,\cdot)\|_{L^{2}(0,L)}\\&\leq T^{\frac{1}{p}}\|\omega\|_{H^{5}(0,L)}\|u\|_{C([0,T];L^2(0,L))}\\&\leq C(T,\|\omega\|_{H^{5}(0,L)})\|u\|_{X(Q_T)}.
\end{split}
\end{equation*}
To bound the last term of \eqref{q'}, note that 
\begin{equation*}
\begin{split}
\left| \int_{0}^{L}f_{2}(t,x)\omega'(x)dx\right|&\leq C(L)\|\omega'\|_{H^{1}_{0}(0,L)}\|f_{2}(t,\cdot)\|_{L^{1}(0,L)}\\
&\leq  C(L)\|\omega\|_{H^{5}(0,L)}\|f_{2}(t,\cdot)\|_{L^{1}(0,L)},
\end{split}
\end{equation*}
since $H^{1}(0,L) \hookrightarrow L^{\infty}(0,L) \cap C[0,L]$. So, last inequality yields that
$$\left\| \int_{0}^{L}f_{2}(t,x)\omega'(x)dx\right\|_{L^{p}(0,T)} \leq C(L,\|\omega\|_{H^{5}(0,L)})\|f_{2}\|_{L^{p}(0,T; L^{1}(0,L))}.
$$
Also, we have
$$\left\| \int_{0}^{L}f_{1}(t,x)\omega(x)dx\right\|_{L^{p}(0,T)} \leq \|\omega\|_{L^{2}(0,L)}\|f_{1}\|_{L^{p}(0,T;L^{2}(0,L))}.$$

To finish this case note that $h_{i}$, for $i=1,2,3,4$ and $h$ belong to $L^{p}(0,T)$. Thus, we have that $q' \in L^{p}(0,T)$, which ensures that $q \in W^{1,p}(0,T)$. Moreover, follows that
\begin{equation*}
\begin{split}
\|q'\|_{L^{p}(0,T)} \leq &\widetilde{C}(T, L, \|\omega\|_{H^{5}(0,L)})\left( \|h\|_{L^{p}(0,T)}  +\|h_{1}\|_{L^{p}(0,T)}+\|h_{2}\|_{L^{p}(0,T)}+\|h_{3}\|_{L^{p}(0,T)} + \|h_{4}\|_{L^{p}(0,T)} \right.\\
&\left.+ \|u\|_{X(Q_{T})}+ \|f_{1}\|_{L^{p}(0,T;L^{2}(0,L))} + \|f_{2}\|_{L^{p}(0,T;L^{1}(0,L))} \right).
\end{split}
\end{equation*}
Thus, estimate \eqref{norma de q'} holds true, showing Case 1.

\vspace{0.2cm}
\noindent\textbf{Case 2.} $p= \infty$.
\vspace{0.2cm}

This case follows noting that 
\begin{equation*}
\begin{split}
\|q'\|_{C([0,T])} \leq &\ C\left( \|u\|_{X(Q_{T})} + \|f_{2}\|_{C(0,T;L^{1}(0,L))} + \|f_{1}\|_{C([0,T];L^{2}(0,L))} \right. \\ & + \left.\|h\|_{C([0,T])} + \|h_{1}\|_{C([0,T])}+\|h_{2}\|_{C([0,T])}+\|h_{3}\|_{C([0,T])} + \|h_{4}\|_{C([0,T])}\right).
\end{split}
\end{equation*}
Thus, Case 2 is achieved and the proof of the lemma is complete.
\end{proof}

The next proposition gives us a relation between $u,\ h$ and $f_1$ and will be the key point to prove the control problems, presented in the next sections.

\begin{lemma}\label{lem1.1} Suppose that  $h \in L^{2}(0,L)$, $f_{1} \in L^{1}(0,T; L^{2}(0,L))$ and $u= S(0,h, f_{1},0)$ mild solution of \eqref{Kw}, then
\begin{equation}\label{1.2}
\int_{0}^{L}|u(t,x)|^{2}dx \leq \int_{0}^{t}|h(t)|^{2}d\tau + 2\int_{0}^{t}\!\!\int_{0}^{L}f_{1}(\tau,x)u(\tau,x)dxdt
\end{equation} 
for all $t \in [0,T]$.
\end{lemma}
\begin{proof}
Pick any function $h \in C_{0}^{\infty}(0,T)$ and consider $f_{1} \in C_{0}^{\infty}(Q_{T})$. Therefore, there exists a smooth solution $u= S(0,h, f_{1},0)$ of \eqref{Kw}. Thus, multiplying   \eqref{Kw} by $2u$, integrating in $[0,L]$ and using the boundary conditions (remembering that $h_1=h_2=h_3=h_4=0$), we get that
\begin{equation*}
\begin{split}
\frac{d}{dt}\int_{0}^{L}|u(t,x)|^{2}dx = &\int_{0}^{L}f_{1}(t,x)u(t,x)dx + |u_{xx}(t,L)|^{2} - |u_{xx}(t,0)|^{2}\\
\leq &\ 2\int_{0}^{L}f_{1}(t,x)u(t,x)dx + |u_{xx}(t,L)|^{2}.
\end{split}
\end{equation*}
So, using the fact that $u_{xx}(t,L) = h(t)$, integrating in $[0,t]$ and taking account that $u(0,\cdot) = 0$ yields
$$
\int_{0}^{L}|u(t,x)|^{2}dx \leq 2\int_{0}^{t}\int_{0}^{L}f_{1}(\tau,x)u(\tau,x)dxd\tau + \int_{0}^{t}|h(\tau)|^{2}d\tau
$$
which implies inequality \eqref{1.2}. By density argument and the continuity of the operator $S$, the result is proved.
\end{proof}

\begin{remarks} We are now giving some remarks.
\begin{itemize}
\item[i.] We are implicitly assuming that $ f_ {2x} \in L ^ {1} (0, T; L ^ {2} (0, L)) $ in the Lemma \ref{1.1}, but it is not a problem, since the function that  we will take for $ f_ {2} $, in our purposes, satisfies that condition.
\vspace{0.1cm}
\item[ii.] When $p= \infty$, in Lemma \ref{lem1.1}, the spaces $L^{p}(0,T)$, $L^{p}(0,T;L^{2}(0,L))$ and $L^{p}(0,T;L^{1}(0,L))$ are replaced by the spaces  $C([0,T])$, $C([0,T];L^{2}(0,L))$ and $ C([0,T];L^{1}(0,L))$, respectively. So, we can obtain $q \in C^1([0,T]).$
\end{itemize} 
\end{remarks}

\section{Boundary control}\label{Sec3}
In this section we are interested in providing answers of overdetermination controllability results for the system \eqref{int1} when the control is acting at the boundary. Precisely, we want to find a control function $h(t)$ acting in the boundary such that the solution of the system in consideration satisfies an overdetermination condition which will take an integral form.

\subsection{Linear result} In this spirit presented above, the first lemma helps to prove a controllability result for the linear case.
\begin{lemma}\label{lema3}
Suppose that $f= \widetilde{h} = u_{0}= 0$ and $\omega \in \mathcal{J}$, with $\omega''(L) \neq 0$ and $\varphi \in \widetilde{W}^{1,p}(0,T)$, for some $p \in [2,\infty].$ Then there exists a unique function $h= \Gamma\varphi \in L^{p}(0,T)$ whose corresponding generalized solution (mild) $u= S(0,h,0,0)$ of \eqref{Kw} satisfies the condition \eqref{IO}. Moreover, the linear operator $$\Gamma: \widetilde{W}^{1,p}(0,T) \longmapsto L^{p}(0,T)$$ is bounded and its norm is nondecreasing with increasing $T$.
\end{lemma}
\begin{proof}  Without loss of generality we consider here $\omega''(L)= 1$, in order to simplify the computations.  First, define the application $\Lambda:  L^{p}(0,T)\longrightarrow \widetilde{W}^{1,p}(0,T)$ as  $$ (\Lambda h)(\cdot)= \int_{0}^{L}u(\cdot,x)\omega(x)dx,$$
with $u = S(0,h,0,0)$, assured by Lemma \ref{H1}. Observe that $ (\Lambda h)(0)=0$ and $\Lambda= (Q \circ S),$ with the functions $Q: X(Q_{T}) \longrightarrow W^{1,p}(0,T)$ defined as $$(Qv)(t)= \int_{0}^{L}v(t,x)\omega(x)dx, \ t \in [0,T]$$
and, in this case, $S$ can be viewed as  $$S: L^{2}(0,L) \longrightarrow X(Q_{T})$$ defined by $u = S(0, h, 0, 0),$ respectively. Since $S$ and $Q$ are linear, we have that  $\Lambda$ is also linear, and thanks to \eqref{norma de q'}, we get that
$$
\|\Lambda(h)\|_{\widetilde{W}^{1,p}(0,T)}=\|(Q\circ S)(h)\|_{\widetilde{W}^{1,p}(0,T)} \leq C(T) \|q'\|_{L^{p}(0,T)} \leq C(T)\|h\|_{L^{2}(0,T)} \quad \forall h \in L^{p}(0,T).
$$
Therefore, $\Lambda$ is continuous. 

Observing that the relation $\varphi= \Lambda h$, for  $h \in L^{p}(0,T)$, clearly means that the function $h$ gives the desired solution of the control problem under consideration. So, our objective is to apply the Banach's theorem to prove that the inverse of the operator $\Lambda$ is continuous. 

To do it, for a fixed function $\varphi \in \widetilde{W}^{1,p}(0,T)$, consider the mapping $A: L^{p}(0,T) \longrightarrow L^{p}(0,T)$ defined by
$$(Ah)(t)= \varphi'(t) - \int_{0}^{L}u(t,x)(\omega'(x) + \omega'''(x) - \omega''''''(x))dx, \ \forall t \in [0,T].$$ Firstly, the following claim holds true.

\begin{claim}\label{claim1}
 $\varphi= \Lambda h$ if and only if $h = Ah$.
\end{claim}

Indeed, if $\varphi = \Lambda h$, then $q(t)= (\Lambda h)(t)= \varphi(t),$ that is,  $q'(t)= \varphi'(t), \ t \in [0,T].$ Therefore, 
\begin{equation*}
 (Ah)(t)=q'(t) - \int_{0}^{L}u(t,x)(\omega'(x) + \omega'''(x) - \omega'''''(x))dx= h(t),
\end{equation*}
thanks to \eqref{q'}. Conversely, if $Ah= h$,  we have
\begin{equation*}
h(t)= \varphi'(t) - \int_{0}^{L}u(t,x)(\omega'(x) + \omega'''(x) - \omega'''''(x))dx.
\end{equation*}
Here, the identity $q'= \varphi'$ holds for the function $q(t) = \Lambda h$ due to the  identity \eqref{q'}. Since,  $\varphi(0)= q(0)=0,$ follows that $\varphi= q$ in $\widetilde{W}^{1,p}(0,T)$, and the Claim \ref{claim1} is proved.

The second claim ensures that:

\begin{claim}\label{claim2}
$A$ is a contraction. 
\end{claim}

In fact, let $ 2 \leq p < \infty$, $\mu_{1}, \mu_{2} \in L^{p}(0,T)$, $u_{1}=S(0,\mu_{1},0,0)$ and $ u_{2}= S(0,\mu_{2},0,0)$  in $X(Q_{T})$. Therefore, 
$$
A\mu_{1} - A\mu_{2}= -\int_{0}^{L}(u_{1} - u_{2})(\omega' + \omega''' - \omega''''')dx.
$$
Moreover, making $u= u_{1}- u_{2}, \ h= \mu_{1}- \mu_{2}$ we have, using \eqref{1.2}, that
\begin{equation}\label{25}
\|u_{1}(t,\cdot)- u_{2}(t,\cdot)\|_{L^{2}(0,L)} \leq \|\mu_{1} - \mu_{2}\|_{L^{2}(0,t)}, \ \forall t \in [0,T].
\end{equation}
Consider $\gamma > 0$. For $t \in [0,T]$, by H\"older inequality, follows that
\begin{equation}\label{25a}
\left|e^{-\gamma t}(A\mu_{1}-A\mu_{2})(t)\right| \leq C(\|\omega\|_{H^{5}(0,L)})e^{-\gamma t}\|u_{1}(t,\cdot) - u_{2}(t,\cdot)\|_{L^{2}(0,L)}.
\end{equation}
The relation \eqref{25} gives,
\begin{equation}\label{25b}
\begin{split}
\|e^{-\gamma t}(A\mu_1- A\mu_2)\|_{L^{p}(0,L)} \leq&\ C(\|\omega\|_{H^{5}(0,L)})\left(\int_{0}^{T}e^{-\gamma pt}\|u_{1}(t,\cdot)- u_{2}(t,\cdot)\|^{p}_{L^{2}(0,L)}dt\right)^{\frac{1}{p}}\\
\leq &\ C(\|\omega\|_{H^{5}(0,L)})\left(\int_{0}^{T}e^{-\gamma pt}\left(\int_{0}^{t}(\mu_{1}(\tau) - \mu_{2}(\tau))^{2}d\tau \right)^{\frac{p}{2}}dt\right)^{\frac{1}{p}}\\
\leq &\ C(p,\|\omega\|_{H^{5}(0,L)})\left(\int_{0}^{T}e^{-\gamma pt}\int_{0}^{t}|\mu_{1}(\tau) - \mu_{2}(\tau)|^{p}d\tau dt\right)^{\frac{1}{p}}\\
\leq &\ C(p,\|\omega\|_{H^{5}(0,L)})\left(\int_{0}^{T}e^{-\gamma pt}|\mu_{1}(\tau) - \mu_{2}(\tau)|^{p}\int_{t}^{T}e^{p\gamma(\tau-t)}d\tau dt\right)^{\frac{1}{p}}\\
\leq&\ C(p,\|\omega\|_{H^{5}(0,L)})\|e^{-\gamma t}(\mu_{1} - \mu_{2})\|_{L^{p}(0,T)}\left(\int_{0}^{T}e^{-\gamma pt}dt\right)^{\frac{1}{p}}\\
\leq&\frac{1}{(p\gamma)^{\frac{1}{p}}} C(p,T,\|\omega\|_{H^{5}(0,L)})\|e^{-\gamma t}(\mu_{1} - \mu_{2})\|_{L^{p}(0,T)}\left(1 - e^{-\gamma pT}\right)^{\frac{1}{p}}\\
=& \ C_1\|e^{-\gamma t}(\mu_{1} - \mu_{2})\|_{L^{p}(0,T)},
\end{split}
\end{equation}
with $C_{1}= C(p,T,\|\omega\|_{H^{5}(0,L)})$.  Therefore, is enough to take  $\gamma= \frac{(2C_{1})^{p}}{p},$ and so $A$ is contraction, showing the Claim \ref{claim2} for the case $p\in[2,\infty)$.

\vspace{0.1cm}

Now, let us analyse the case  $p= \infty$. Using \eqref{25a}, yields that
\begin{equation}\label{25c}
\begin{split}
\sup_{t \in [0,T]}e^{-\gamma t}\left|(A\mu_{1}- A\mu_{2})(t)\right|= & \ C( \|\omega\|_{H^{5}(0,L)})\sup_{t \in [0,T]}e^{-\gamma t}\|u_{1}(t,\cdot)- u_{2}(t,\cdot)\|_{L^{2}(0,L)}\\
\leq&\ C( \|\omega\|_{H^{5}(0,L)})\sup_{t \in [0,T]}e^{-\gamma t}\|\mu_{1} - \mu_{2}\|_{L^{2}(0,t)}\\
\leq& \ C(\|\omega\|_{H^{5}(0,L)})\sup_{t \in [0,T]}\left(\int_{0}^{t}e^{2\gamma(\tau - t)}|\mu_{1}(\tau) - \mu_{2}(\tau)|^{2}d\tau\right)^{\frac{1}{2}}\\ 
\leq&\ C(\|\omega\|_{H^{5}(0,L)})\|e^{-\gamma t}(\mu_{1} - \mu_{2})\|_{L^{\infty}(0,T)}\sup_{t \in [0,T]}\left(\frac{1}{2\gamma}[1 - e^{-2\gamma t}]\right)^{\frac{1}{2}}
\end{split}
\end{equation}
so taking  $\gamma = 2C_{1}^{2}$ yields that $A$ is a contraction, showing the claim for $p=\infty$. This analysis ensures that the mapping $A$ is a contraction, and Claim \ref{claim2} is achieved.

Therefore, for each function $\varphi \in \widetilde{W}^{1,p}(0,T)$, there exists a unique function $h \in L^{p}(0,T)$ such that $h= A(h)$, that is, $\varphi= \Lambda(h)$. It follows that operator $\Lambda$ is invertible, and so, its inverse  $\Gamma:=\Lambda^{-1}:L^{p}(0,T) \longmapsto \widetilde{W}^{1,p}(0,T)$ is continuous thanks to the Banach theorem. In particular,
$$
\|\Gamma(\varphi)\|_{L^{p}(0,T)} \leq C(T)\|\varphi'\|_{L^{p}(0,T)}.
$$
By a standard argument we continuously extend the function $\varphi$ by the constant $\varphi(T)$ in $(T,T_{1})$ with the previous inequality still valid in $(0,T_{1})$ with $C(T) \leq C(T_{1})$, therefore the operator $\Gamma$ in nondecreasing with increasing T, proving the result.
\end{proof}

With the previous result in hand, now let us show a controllability result for the linear case.

\begin{theorem}\label{caso linear 1}
Consider $p \in [2, \infty]$, $\varphi \in W^{1,p}(0,T)$, $u_{0} \in L^{2}(0,L)$, $\widetilde{h} \in \mathcal{H}$, with $ h_{i}\in L^{p}(0,T)$, for $i=1,2,3,4$ and $f= f_{1} + f_{2x},$ where $f_{1} \in L^{p}(0,T; L^{2}(0,L))$. Moreover, if $f_{2} \in L^{p}(0,T; L^{1}(0,L))$ such that $f_{2x} \in L^{1}(0,T; L^{2}(0,L))$ and $\omega \in \mathcal{J}$, with $\omega''(L)\neq 0$, satisfies \eqref{OIa}, then there exists a unique function $h \in L^{p}(0,T)$ such that the mild solution $u= S(u_{0}, h, f_{1} + f_{2x}, \widetilde{h})$ of \eqref{Kw} verifies the overdetermination condition \eqref{IO}.
\end{theorem}
\begin{proof}
Here, consider $$S: L^{2}(0,L)\times L^{2}(0,L) \times L^{1}(0,T; L^{2}(0,L)) \times \mathcal{H} \longrightarrow X(Q_{T}),$$
with $\widehat{u}= S(u_{0}, 0, f_{1} + f_{2x}, \widetilde{h})$ mild solution of the system \eqref{Kw}.  Now, consider the application $\widehat{\varphi}= \varphi - Q(\widehat{u})$, where $\varphi \in \widetilde{W}^{1,p}(0,T).$ Lemma \ref{1.1} together with \eqref{OIa}, ensures that $\widehat{\varphi} \in \widetilde{W}^{1,p}(0,T)$. Thus, Lemma \ref{1.2} guarantees the existence of a unique $\Gamma \widehat{\varphi}= h \in L^{p}(0,T)$ such that the solution $v= S(0,h,0,0)$ of \eqref{Kw}  satisfies
$$
\int_{0}^{L}v(t,x)\omega(x)dx= \widehat{\varphi}(t), \  t \in [0,T].
$$

Thus, if $u= \widehat{u} + v= S(u_{0},h, f_{1} + f_{2x}, \widetilde{h})$, we have that $u$ is solution of \eqref{Kw} satisfying
$$ 
\int_{0}^{L}u(t,x)\omega(x)dx= \varphi(t), \quad t \in [0,T].
$$ 
So, the proof of the theorem is complete.
\end{proof}

\subsection{Nonlinear result} In this section we are able to prove the first main result of this manuscript.

\begin{proof}[Proof of Theorem \ref{main}] In the assumptions of Theorem \ref{caso linear 1} consider $f_{1}= f$ and $f_{2}= \frac{v^{2}}{2}$, with $v \in X(Q_{T})$ and $f \in L^{1}(0,T;L^{2}(0,L))$. Note that $v^2\in \ C(0,T;L^{1}(0,L)) \hookrightarrow L^{p}(0,T;L^{1}(0,L))$, for $p \in [1, \infty]$. So, using the following inequality 
\begin{equation*}
\displaystyle \sup_{x \in [0,L]}|g(x)|^{2} \leq C(L)\bigl(\|g'\|_{L^{2}(0,L)}\|g\|_{L^{2}(0,L)} + \|g\|^{2}_{L^{2}(0,L)}\bigr).
\end{equation*}
we have $v^2\in L^2(Q_T)$ and 
\begin{equation*}
\begin{split}
\|v^{2}\|_{L^{2}(Q_{T})}\leq&\int_0^T\left(\sup_{x \in [0,L]}|v(t,\cdot)|^{2} \int_{0}^{L}|v(t,x)|^{2}dxdt\right)^{\frac{1}{2}}\\
\leq& \ C(L)\left(\int_{0}^{T}\|v(t,\cdot)\|^{3}_{L^{2}(0,L)}.\|v_{xx}(t,\cdot)\|_{L^{2}(0,L)}dt + \int_{0}^{T}\|v(t,\cdot)\|^{4}_{L^{2}(0,L)}dt\right)^{\frac{1}{2}}\\
\leq& \ C(L)\left(\displaystyle\sup_{t \in [0,T]}\|v(t,\cdot)\|^{3}_{L^{2}(0,L)}\int_{0}^{T}\|v_{xx}(t,\cdot)\|_{L^{2}(0,L)}dt + T\sup_{t \in [0,T]}\|v(t,\cdot)\|^{4}_{L^{2}(0,L)}\right)^{\frac{1}{2}}\\
\leq&\ C(L)\left(\sup_{t \in [0,T]}\|v(t,\cdot)\|^{\frac{3}{2}}_{L^{2}(0,L)}\left(\int_{0}^{T}\|v_{xx}(t,\cdot)\|_{L^{2}(0,L)}dt\right)^{\frac{1}{2}} + T^{\frac{1}{2}}\sup_{t \in [0,T]}\|v(t,\cdot)\|^{2}_{L^{2}(0,L)}\right)\\
\leq&\ C(L)\left(\sup_{t \in [0,T]}\|v(t,\cdot)\|^{\frac{3}{2}}_{L^{2}(0,L)}\left(\int_{0}^{T}\|v_{xx}(t,\cdot)\|_{L^{2}(0,L)}dt\right)^{\frac{1}{2}} + T^{\frac{1}{2}}\|v\|^{2}_{X(Q_{T})}\right)\\
\leq & \ C(L) \left(\sup_{t \in [0,T]}\|v(t,\cdot)\|^{\frac{3}{2}}_{L^{2}(0,L)}\left(T^{\frac{1}{2}}\|v_{xx}\|_{L^{2}(0,T;L^{2}(0,L))}\right)^{\frac{1}{2}}+ T^{\frac{1}{2}}\|v\|^{2}_{X(Q_{T})}\right)\\
\leq & \ C(L) \left(T^{\frac{1}{4}}\sup_{t \in [0,T]}\|v(t,\cdot)\|_{L^{2}(0,L)}\left( \sup_{t \in [0,T]}\|v(t,\cdot)\|_{L^{2}(0,L)}\|v_{xx}\|_{L^{2}(Q_{T})} \right)^{\frac{1}{2}}+ T^{\frac{1}{2}}\|v\|^{2}_{X(Q_{T})}\right)\\
\leq &\ C(L)\left( T^{\frac{1}{4}}\|v\|_{X(Q_{T})}\left( \sup_{t \in [0,T]}\|v(t,\cdot)\|_{L^{2}(0,L)} + \|v_{xx}\|_{L^{2}(Q_{T})} \right)+ T^{\frac{1}{2}}\|v\|^{2}_{X(Q_{T})}\right)\\
=& \ C(L) \left(T^{\frac{1}{4}} +T^{\frac{1}{2}} \right)\|v\|^{2}_{X(Q_{T})},
\end{split}
\end{equation*}
showing that 
\begin{equation}\label{43}
\|v^{2}\|_{L^{2}(0,T;L^{2}(0,L))}= \|v^{2}\|_{L^{2}(Q_{T})}\leq C(L)(T^{\frac{1}{2}} + T^{\frac{1}{4}})\|v\|^{2}_{X(Q_{T})}.
\end{equation}

On the space $X(Q_T)$ define the application $\Theta: X(Q_{T}) \longrightarrow X(Q_{T})$ by  $$ \Theta v = S\left(u_{0}, \Gamma\left(\varphi - Q(S(u_{0},0, f - vv_{x}, \widetilde{h}))\right), f - vv_{x}, \widetilde{h}\right).$$
Let $\varphi \in W^{1,p}(0,T)$, $u_{0} \in L^{2}(0,L)$, $ \widetilde{h} \in \mathcal{H}$ be given such that $h_{i} \in L^{p}(0,T)$, for $i=1,2,3,4$  and  $ \ f \in L^{1}(0,T; L^{2}(0,L))$.  Applying the results of Lemma \ref{Zhang 2}, for $s=0$, Theorem \ref{caso linear 1} and inequality \eqref{43} we have, for $p=2$, that 
\begin{equation*}
\begin{split}
\|\Theta v\|_{X(Q_{T})}\leq&\ C(T)\left(\|u_{0}\|_{L^{2}(0,L)} + \|f\|_{L^{1}(0,T; L^{2}(0,L))} +  \|\widetilde{h}\|_{\mathcal{H}} \right.\\
&+ \left.\left\|vv_{x}\right\|_{L^{1}(0,T; L^{2}(0,L))} + \left\|\varphi - Q(S(u_{0},0, f - vv_{x}, \widetilde{h}))\right\|_{\widetilde{W}^{1,p}(0,T)} \right)\\
\leq&\ C(T)\left(\|u_{0}\|_{L^{2}(0,L)} + \|f\|_{L^{1}(0,T; L^{2}(0,L))} +  \|\widetilde{h}\|_{\mathcal{H}} \right.\\
&+\left.\|vv_{x}\|_{L^{1}(0,T; L^{2}(0,L))} + \|\varphi'\|_{L^{2}(0,T)} + \|q'\|_{L^{2}(0,T)}\right)\\
\leq&\ C(T)\left(\|u_{0}\|_{L^{2}(0,L)} + \|f\|_{L^{1}(0,T; L^{2}(0,L))} +  \|\widetilde{h}\|_{\mathcal{H}} + \|\varphi'\|_{L^{2}(0,T)} \right.\\
&+\left.\|f\|_{L^{2}(0,T; L^{2}(0,L))} + \|vv_{x}\|_{L^{1}(0,T; L^{2}(0,L))} + \left\|\frac{v^{2}}{2}\right\|_{L^{2}(0,T;L^{1}(0,L))}\right)\\
\leq & \ C(T)\left(c_{0} + T^{\frac{1}{4}}\right)\|v\|^{2}_{X(Q_{T})},
\end{split}
\end{equation*}
and in a similar way
\begin{equation*}
\| \Theta v_{1} - \Theta v_{2}\|_{X(Q_{T})} \leq C(T) T^{\frac{1}{4}}(\|v_{1}\|_{X(Q_{T})} + \|v_{1}\|_{X(Q_{T})})\|v_{1} - v_{2} \|_{X(Q_{T})},
\end{equation*}
where $$c_{0}= \|u_{0}\|_{L^{2}(0,L)} + \|f\|_{L^{2}(0,T; L^{2}(0,L))} +  \|\widetilde{h}\|_{\mathcal{H}} + \|\varphi'\|_{L^{2}(0,T)}$$ and the constant $C(T)$ is nondecreasing with increasing $T$.

Fix $c_{0}$ and  consider $T_{0} > 0$ such that $8C(T_{0})^{2}T_{0}^{\frac{1}{4}}c_{0} \leq 1$ and then, for any $T \in (0, T_{0}]$, we can choose  $$r \in \left[2C(T)c_{0}, \frac{1}{\left(4C(T)T^{\frac{1}{4}}\right)}\right].$$ By the other hand, for a fixed $T > 0$, pick $$r= \frac{1}{\left(4C(T,L)T^{\frac{1}{4}}\right)}$$ and $$c_{0} \leq \gamma= \frac{1}{\left(8C(T)^{2}T^{\frac{1}{4}}\right)},$$ so in both cases
$$C(T)c_{0} \leq \frac{r}{2} \quad  \text{and}\quad C(T)T^{\frac{1}{4}}r \leq \frac{1}{4},$$
thus $\Theta$ is a contraction on $B(0,r) \subset X(Q_{T})$. In this way, there exists a unique fixed point $$u=\Theta u\in X(Q_{T})$$ satisfying \eqref{int1} and the integral condition \eqref{1.2} when  $$h = \Gamma\left(\varphi - Q(S(u_{0},0,f -uu_{x},\widetilde{h}))\right).$$ As the uniqueness can be obtained in the standard way, Theorem \ref{main} is proved.
\end{proof}

\section{internal control}\label{Sec4}
This section is dedicated to prove the internal controllability result for system \eqref{int1} when $f$ assumes a special form, namely $f(t,x)=f_0(t)g(t,x)$. First, we prove that the linear system associated to \eqref{int1} is controllable in the sense proposed in the introduction, finally, we extend this result for the full system using a fixed point theorem, as made in the previous section. 

\subsection{Linear result} The next lemma is a key point to prove one of the main results of this manuscript and can be read as follows.
\begin{lemma}\label{lss} Assuming that  $h= \widetilde{h}= u_{0}= 0$ in the system \eqref{Kw}, for $g \in C(0,T; L^{2}(0,L))$ and $\omega \in \mathcal{J}$ be given such that
\begin{equation}\label{g0}
\left|\int_{0}^{L}g(t,x)\omega(x)dx\right| \geq g_{0} > 0, \ \forall \ t \in [0,T],
\end{equation}
and $\varphi \in \widetilde{W}^{1,p}(0,T)$, for some $p \in [1, \infty]$, there exists a unique function $f_{0}= \Gamma(\varphi) \in L^{p}(0,T)$ such that the solution $u = S(0,0,f_{0}g,0)$ of \eqref{Kw} satisfies the overdeternination condition \eqref{IO}. Moreover, $$\Gamma : \widetilde{W}^{1,p}(0,T) \longmapsto L^{p}(0,T)$$ is a linear bounded operator  and its norm is nondecreasing with increasing T.
\end{lemma}
\begin{proof}With this hypothesis in hand, define the following linear application $$G: L^{p}(0,T)  \longrightarrow L^{1}(0,T; L^{2}(0,L))$$ by $G(f_{0})= f_{0}g,$ which satisfies 
$$\|G(f_{0})\|_{L^{1}(0,T;L^{2}(0,L))} \leq T^{\frac{p-1}{p}}\|g\|_{C([0,T];L^{2}(0,L))}\|f_{0}\|_{L^{p}(0,T)}$$
Now, considering the mapping $\Lambda= Q\circ S \circ G : L^{p}(0,T) \longrightarrow \widetilde{W}^{1,p}(0,T)$ as  $$ (\Lambda f_{0})(t)=\int_{0}^{L}u(t,x)\omega(x)dx,$$ where $u= S(0,0,f_{0}g,0)$, since $Q$, $S$ and $G$ are linear and bounded operators, we have that $\Lambda$ is a bounded linear operator. Additionally, using Lemma \ref{1.1}  and the continuity of the operator $S$, $\Lambda$ acts boundedly from the spaces  $L^{p}(0,T)$ to the space $\widetilde{W}^{1,p}(0,T)$.

Note that $\varphi=\Lambda f_0$, for $f_0 \in L^p(0,T)$, means that the function $f_0$ gives the desired solution of our control problem. So, with this in hand define the operator $A : L^{p}(0,T) \longrightarrow L^{p}(0,T)$ by \begin{equation*} (Af_{0})(t)= \frac{\varphi'(t)}{g_{1}(t)} - \frac{1}{g_{1}(t)}\int_{0}^{L}u(t,x)(\omega' + \omega''' - \omega''''')dx, \end{equation*} where $u= S(0,0,f_{0}g,0)$ and $$ g_{1}(t)= \int_{0}^{L}g(t,x)\omega(x)dx,$$ for all $t \in [0,T]$.  Thus, as done in Lemma \ref{lema3} we have that $\Lambda f_{0}= \varphi$  if and only if $f_{0}= Af_{0}$.  

Now, we are concentrating on proving the following.

\begin{claim}\label{claim3}
$A$ is a contraction.
\end{claim}

For the case when $1 \leq p < \infty $, consider $f_{01}$ and $f_{02}$ in $L^{p}(0,T)$,  $u_{1}= S(0,0,f_{01}g,0)$ and $u_{2}= S(0,0,f_{02}g,0)$. Thus,
$$
Af_{01} - Af_{02}= -\frac{1}{g_{1}}\int_{0}^{L}(u_{1} - u_{2})(\omega' + \omega''' - \omega''''')dx.
$$
Moreover, rewrite the following functions as  $u= u_{1}- u_{2}$ and $f_{1}= f_{01}- f_{02}$, thanks to the inequality \eqref{1.2}, holds that
\begin{equation}\label{123}
\|u_{1}(t,\cdot)- u_{2}(t,\cdot)\|_{L^{2}(0,L)} \leq 2\|g\|_{C(0,T; L^{2}(0,L))}\|f_{01} - f_{02}\|_{L^{1}(0,t)}, \ \forall t \in [0,T].
\end{equation}

Let $\gamma > 0$, for $p<+\infty$ in the analog way as we did in \eqref{25b}, we have 
\begin{equation*}
\begin{split}
\|e^{-\gamma t}(Af_{01}- Af_{02})\|_{L^{p}(0,T)} \leq& \frac{1}{g_{0}}\|\omega\|_{H^{5}(0,L)}\left(\int_{0}^{T}e^{-\gamma pt}\|u_{1}(t,\cdot)- u_{2}(t,\cdot)\|^{p}_{L^{2}(0,L)}dt\right)^{\frac{1}{p}}\\
\leq&\ C(p)\left(\int_{0}^{T}e^{-\gamma pt}\int_{0}^{T}|f_{01}(\tau) - f_{02}(\tau)|^{p}d\tau dt\right)^{\frac{1}{p}}\\
\leq& \frac{2^{\frac{1}{p}}C(p)}{(\gamma p)^{\frac{1}{p}}}\|f_{01} - f_{02}\|_{L^{p}(0,T)}\left(\frac{1 - e^{-\gamma pT}}{2}\right)^{\frac{1}{p}}\\
\leq&\ \frac{C_{1}}{(\gamma p)^{\frac{1}{p}}}\|e^{-\gamma t}(f_{01} - f_{02})\|_{L^{p}(0,T)}
\end{split}
\end{equation*}
where $C_{1} = C_{1}(T,p,\|\omega\|_{H^{5}(0,L)}, g_{0}, \|g\|_{C(0,T; L^{2}(0,L))})$. So, just take $\gamma = \frac{(2C_{1})^{p}}{p}$ and  $A$ is a contraction in this case.

Now, consider $p= \infty$. For $\gamma > 0$, we get similarly to what was done \eqref{25c} that
\begin{equation*}
\begin{split}
\sup_{t \in [0,T]}e^{-\gamma t}\left|((Ah_{1})(t)- (Ah_{2})(t))\right)|\leq &\ 
2C(g_{0},\|g\|,\|\omega\|_{H^{5}(0,L)})\sup_{t \in [0,T]}e^{-\gamma t}\|f_{01} - f_{02}\|_{L^{1}(0,t)}\\
\leq&\ 2C(g_{0},\|g\|, \|\omega\|_{H^{5}(0,L)})\sup_{t \in [0,T]}\int_{0}^{t}e^{\gamma(\tau - t)}|f_{01}(\tau) - f_{02}(\tau)|d\tau\\
\leq & \ 2C(g_{0},\|g\|, \|\omega\|_{H^{5}(0,L)})\|f_{01} - f_{02}\|_{L^{\infty}(0,T)}\sup_{t \in [0,T]}\frac{1}{\gamma}[1 - e^{-\gamma t}]\\
\leq &\ \frac{1}{\gamma}2C(g_{0},\|g\|, \|\omega\|_{H^{5}(0,L)})\|f_{01} - f_{02}\|_{L^{\infty}(0,T)}\\
\leq & \frac{C_{1}}{\gamma}\|f_{01} - f_{02}\|_{L^{\infty}(0,T)},
\end{split}
\end{equation*}
where $C_{1} = C_{1}(T,\|\omega\|_{H^{5}(0,L)}, g_{0}, \|g\|_{C(0,T; L^{2}(0,L))})$. Thus, if $\gamma = 2C_{1}$ we have $A$ is contractions, finishing the case $p=+\infty$, proving Claim \ref{claim3}.

Therefore, for each  $\varphi \in \widetilde{W}^{1,p}(0,T)$, there exists a unique  $f_{0} \in L^{p}(0,T)$ such that $f_{0}= A(f_{0})$, i.e., $\varphi= \Lambda(f_{0})$.  It follows that $\Lambda$ is invertible, and its inverse $\Gamma :L^{p}(0,T) \longmapsto \widetilde{W}^{1,p}(0,T)$ is a continuous operator thanks to the Banach's theorem. Additionally, we have 
$$ \|\Gamma(\varphi)\|_{L^{p}(0,T)} \leq C(T)\|\varphi'\|_{L^{p}(0,T)}.$$
The end of the proof follows in the same way as in Lemma \ref{lema3}, and so, the proof is complete. 
\end{proof}

Let us now enunciate a result concerning the internal controllability for the linear system. The result is the following one. 

\begin{theorem}\label{caso linear 2} Assume that $p \in [1, \infty]$, $u_{0} \in L^{2}(0,L)$, $ h \in L^{max\{2,p\}}(0,T)$, $\widetilde{h} \in \mathcal{H},$ with $ h_{i} \in L^{p}(0,T),$ for $i=1,2,3,4$ and  $f_{2} \in L^{p}(0,T; L^{1}(0,L))$ such that $f_{2x} \in L^{1}(0,T; L^{2}(0,L))$. If $ g \in C([0,T]; L^{2}(0,L))$, $\omega \in \mathcal{J}$, $\omega''(L)\neq 0$ and $ \varphi \in W^{1,p}(0,T)$ satisfies \eqref{OIa} and
$$
\left| \int_{0}^{L}g(t,x)\omega(x)dx\right| \geq g_{0} > 0, \ \forall t \in [0,T]
$$
then there exists a unique function $f_{0} \in L^{p}(0,T)$ such that the solution $u= S(u_{0}, h, f_{0}g + f_{2x}, \widetilde{h})$ of \eqref{Kw} satisfies
$$
\int_{0}^{L}u(t,x)\omega(x)dx= \varphi(t), \ t \in [0,T].
$$
\end{theorem}
\begin{proof}
Pick $\widehat{u}= S(u_{0}, h, -f_{2x}, \widetilde{h})$ solution of \eqref{Kw} with $f=-f_{2x}$. Now, consider $\widehat{\varphi}= \varphi - Q(\widehat{u})$ with $\varphi \in W^{1,p}(0,T).$ By Lemma \ref{1.1}, together with \eqref{OIa}, follows that 
$\widehat{\varphi} \in \widetilde{W}^{1,p}(0,T).$ Therefore, due to the Lemma \ref{lss}, there exists a unique  $\Gamma \widehat{\varphi}= f_{0} \in L^{p}(0,T)$ such that the solution  $v= S(0,0,f_{0}g,0)$ of \eqref{Kw} with $f=f_0g$ satisfies 
$$
\int_{0}^{L}v(t,x)\omega(x)dx= \widehat{\varphi}(t), \  t \in [0,T].
$$
Thus, taking  $u= \widehat{u} + v= S(u_{0},h, f_{0}g - f_{2x}, \widetilde{h})$, we have that  $u$ solution of \eqref{Kw} have the following property
$$\int_{0}^{L}u(t,x)\omega(x)dx= \varphi(t),$$
for $t \in [0,T]$, showing the result.
\end{proof}

\subsection{Nonlinear result}In this section we are able to prove the second main result of this article.

\begin{proof}[Proof of Theorem \ref{main1}] In the assumption of Theorem \ref{caso linear 2}, pick $f_2=-\frac{v^2}{2}$ for an arbitrary $v\in X(Q_T)$. Now, define the mapping $\Theta: X(Q_{T}) \longrightarrow X(Q_{T})$ as follows  $$\Theta v = S\left(u_{0},h, \Gamma\left(\varphi - Q(S(u_{0},h, - vv_{x}, \widetilde{h}))\right)g - vv_{x}, \widetilde{h}\right).$$
In the same way as done in the proof of Theorem \ref{main}, we have
\begin{equation*}
 \|\Theta v\|_{X(Q_{T})} \leq C(T)\left(c_{0} + T^{\frac{1}{4}}\|v\|^{2}_{X(Q_{T})}\right)
\end{equation*}
and
\begin{equation*}
\| \Theta v_{1} - \Theta v_{2}\|_{X(Q_{T})} \leq C(T)T^{\frac{1}{4}}(\|v_{1}\|_{X(Q_{T})} + \|v_{1}\|_{X(Q_{T})})\|v_{1} - v_{2} \|_{X(Q_{T})}.
\end{equation*}
With this in hand we can proceed as the Theorem \ref{main} to conclude that $\Theta$ is a contraction and there exists a unique fixed point $u \in X(Q_{T})$ such that  $f_{0} = \Gamma\bigl(\varphi - Q(S(u_{0},h, -uu_{x},\widetilde{h}))\bigr)$.
\end{proof}

\subsection*{Acknowledgments:} 

R. de A. Capistrano--Filho was supported by CNPq grant 408181/2018-4, CAPES-PRINT grant  88881.311964/2018-01, MATHAMSUD grants  88881.520205/2020-01, 21-MATH-03 and Propesqi (UFPE). L. S. de Sousa acknowledges support from CAPES-Brazil and CNPq-Brazil. This work is part of the PhD thesis of L. S. de Sousa at Department of Mathematics of the Universidade Federal de Pernambuco.

\end{document}